\documentclass[10pt]{amsart}

\usepackage {amsmath, amsopn}
\usepackage[applemac]{inputenc}
\usepackage {amsfonts, amssymb}
\usepackage {latexsym}

\newtheorem {theorem} {Theorem} [section]
\newtheorem {lemma} {Lemma} [section]
\newtheorem {prop} {Proposition} [section]

\newtheorem{preremark}{Remark}[section]
\newenvironment{remark}%
  {\begin{preremark}\rm}{\end{preremark}}
   \newtheorem{preremark1}{Example}[section]

     \newtheorem{preremark2}{Definition}[section]
\newenvironment{defn}
  {\begin{preremark2}\rm}{\end{preremark2}}

\begin{document}



\title[Weak solutions for a bioconvection model]{Weak solutions for a bioconvection model related to Bacillus subtilis}

\author[D.Vorotnikov]{Dmitry Vorotnikov}

\dedicatory{CMUC, Department of Mathematics, University of Coimbra\\ 3001-454 Coimbra, Portugal\\ mitvorot@mat.uc.pt}

\thanks{Research partially supported by CMUC and FCT (Portugal), through
European program COMPETE/FEDER}




\begin{abstract} We consider the initial-boundary value problem for the coupled
Navier-Stokes-Keller-Segel-Fisher-Kolmogorov-Petrovskii-Piskunov
system in two- and three-dimensional domains. The problem describes oxytaxis and growth of Bacillus
subtilis in moving water. We prove existence of global weak solutions to the problem. We distinguish between two cases determined by the cell diffusion term and the space dimension, which are referred as the supercritical and subcritical ones. At the first case, the choice of the growth function enjoys wide range of possibilities: in particular, it can be zero. Our results are new even at the absence of the growth term. At the second case, the restrictions on the growth function are less relaxed: for instance, it cannot be zero but can be Fisher-like. In the case of linear cell diffusion, the solution is regular and unique provided the domain is the whole plane. In addition, we study the long-time behaviour of the problem, find dissipative estimates, and construct attractors. 
\end{abstract}

\maketitle

\section {Introduction}
\newcommand{\wx}{\langle x \rangle}
\newcommand{\sgn}{\text{sgn}}
\newcommand {\R} {\mathbb{R}}
\newcommand {\E} {\mathbf{E}}

\def\fr#1#2{\frac{\partial #1}{\partial #2}}

Let us fix a number $T>0$ and a domain $\Omega\subset \R^{d}$, with $d=2,3$, which can be a bounded open set locally located on one side of its $C^2$-smooth boundary $\partial \Omega$ or the whole space  $\R^{d}$ itself. In the cylinder $Q_{T} =
(0,\, T) \times \Omega$, we
consider the following set of equations
\begin{equation}\label{ksn1}
\partial_t n + u \cdot \nabla  n - \Delta (n^m)= -\nabla\cdot (\chi (c) n \nabla c)+f(n),\end{equation} \begin{equation} \label{ksn2}
\partial_t c + u \cdot \nabla c-\Delta c =-k(c) n,\end{equation} \begin{equation} \label{ksn3}
\partial_t u + u\cdot \nabla u -\Delta u +\nabla p=-n \nabla
\phi,\end{equation} \begin{equation} \label{ksn4}
\nabla \cdot u=0.
\end{equation}

Here $c(t,\,x) : Q_{T} \rightarrow \R$, $n(t,\,x) : Q_{T}
\rightarrow \R$, $u(t,\, x) : Q_{T} \rightarrow \R^{d}$ and
$p(t,x) :  Q_{T} \rightarrow \R$ are the oxygen concentration,
cell concentration, fluid velocity, and hydrostatic pressure,
respectively. The scalar functions $k$, $\chi$ and $f$ determine the oxygen
consumption rate, 
chemotactic sensitivity, and bacterial growth, resp., $\phi:Q_{T} \rightarrow \R$ is the potential produced by the action of physical forces on the cells, and $m\geq 1$ is the nonlinear diffusion exponent. The cases $m=1$ and $f\equiv 0$ are not excluded.

The system is complemented with the no-flux boundary conditions for $n^m$ and $c$, and the no-slip condition for $u$,
\begin{equation} \label{ksnb} \fr {n^m(x)} \nu = 0,\ \fr {c(x)} \nu = 0,\ u(x)=0,\ x\in \partial \Omega,\end{equation}
and with the initial conditions \begin{equation} \label{ksn5} n(0,x)=n_0(x),\ c(0,x)=c_0(x),\ u(0,x)=u_0(x),\ x\in \Omega.\end{equation}

Cf. \cite{w12,w10,pk92,hp05,tw12,ll11,ckl11,tuval,mt96,mt02,mt06,wh07,npr08,l10,flm10,hp09,dlm10}.

\section{Preliminaries}

The symbol $C$ will stand for a generic positive constant that can take different values in different lines, whereas $K_i$, $i=1,2,\dots$, will be fixed positive constants.


Denote $(u,v)=\int_\Omega u(x)v(x)\,dx$, $uv\in L_1$, and $\|u\|=\sqrt{(u,u)}$ (the norm in $L_2$). $\|u\|_l$ stands for the norm in $H^l$, $l\in \mathbb{N}$.

We set $\wx=\sqrt{1+|x|^2}$ in the case $\Omega=\R^d$, and $\wx=1$ for bounded $\Omega$. 

The symbols $C (\mathcal{J}; E) $, $C_w(\mathcal{J}; E)$, $L_2
(\mathcal{J}; E) $ etc. denote the spaces of continu\-ous, weakly continu\-ous, quadratically integrable etc. functions on an interval
$\mathcal{J}\subset \mathbb {R} $ with
values in a Banach space $E $. A pre-norm in the Frechet space $C ([0, +
\infty); E) $ may be defined by the formula
$$ \| v \| _{C ([0, +\infty);E)} =\sum\limits _ {i=1} ^ {+ \infty} 2 ^ {-i} \frac {\|v \| _ {C ([0, i]; E)}} {1 + \| v \| _ {C ([0, i]; E)}}. $$

\begin{defn}\label{defweak}
A triple $(n,\, c,\, u)$ is a \emph{weak
solution} to the problem \eqref{ksn1}--\eqref{ksn5} provided
$$n\geq 0,\ c\geq 0,$$ 
$$
n \in L_{\infty}(0, T; L_{1})\cap L_{2}(0, T; L_{2}) \cap L_{m}(0, T; L_{m}) \cap W^{1}_{1}(0, T; (W^1_\infty)^*),
$$
$$
f(n) \in L_{1}(0, T; L_{1}),\ \nabla(n^m)\in L_{1}(0, T; L_{1}), 
$$
$$
c \in L_{\infty} (0,T; L_{\infty} \cap
H^{1} )\cap 
L_2 (0, T; H^{2})\cap H^{1}(0, T; (H^{1})^*),
$$
$$
u \in L_{\infty}(0, T; L_{2}) \cap L_2(0,
T; V) \cap W^{1}_{1}(0, T; V^*), 
$$
and for any test functions $\zeta\in W^1_\infty,\theta\in H^1, \psi \in V$ one has 
\begin{equation}\label{weak1}
\frac {d}{dt} (n,\zeta) - (u n,\nabla \zeta) +(\nabla(n^m),\nabla\zeta)-(\chi (c) n \nabla c,\nabla \zeta)=(f(n),\zeta),
\end{equation}

\begin{equation}\label{weak2}
\frac {d}{dt} (c,\theta) - (u c,\nabla \theta)+ (\nabla c,\nabla\theta) +(k(c) n,\theta)=0,
\end{equation}

\begin{equation}\label{weak3}
\frac {d}{dt} (u,\psi) - \sum\limits^d _ {i,j=1} \left(u_i u_j, \frac {\partial\psi_j} {\partial x_i}\right) + (\nabla u,\nabla \psi) +(n \nabla
\phi,\psi)=0
\end{equation}
a.e. on $(0,T)$, and the equalities \eqref{ksn5} hold in the spaces $(W^1_\infty)^*$, $(H^{1})^*$, $V^*$, resp.


\end{defn}

\section{The supercritical case}

\begin{theorem} \label{mb1} Let $m>\frac {d+1} 3$. Let $\phi\in L_1(0,T; L_{1,loc})$ with $\nabla\phi\in L_2(0,T; L_\infty).$ Let $k,\chi$ and $f$ be continuously differentiable functions, $\chi'\geq 0$, $k\geq 0$, $k(0)= 0$, $f(0)\geq 0$ (but $f(0)= 0$ for $\Omega=\R^d$) and \begin{equation} \label{asf1} f(y)\leq f(0)+Cy \end{equation} for $y\geq 0$. 

Let $n_0\in L_1 \cap L_{\max(1,m/2)}$, $n_0 \ln n_0 \in L_1$, $ \langle\cdot\rangle n_0(\cdot)\in L_1,$ $c_0 \in H^1\cap L_{\infty}$, $n_0 \geq 0$, $c_0 \geq 0$, $u_0 \in H$. Then problem \eqref{ksn1}--\eqref{ksn5} possesses a weak solution. 
\end{theorem}

\begin{proof} Let us show that a solution $(n,c,u)$ to \eqref{ksn1}--\eqref{ksn5} satisfies the following formal a priori bound:
\begin{multline} \label{bound1} \|u\|_{L_{\infty}(0, T; L_{2})} +  \|n\ln n\|_{L_{\infty}(0, T; L_{1})} +\| \langle\cdot\rangle n\|_{L_{\infty}(0, T; L_{1})} + \|\nabla c\|_{L_{\infty}(0, T; L_{2})} \\
+
\| \nabla u\|_{L_{2}(0, T; L_{2})} + \|n\|_{L_{2}(0, T; L_{2})}+ \|n^{m/2}\|_{L_{2}(0, T; L_{2})}+ \|\nabla n^{m/2}\|_{L_{2}(0, T; L_{2})} \\ +\|f(n)\|_{L_{1}(0, T; L_{1})}+\|f(n)\ln n\|_{L_{1}(0, T; L_{1})}+ \|
c\|_{L_{2}(0, T; H^{2})}  \le C.  \end{multline}

Letting $\zeta\equiv 1$ in \eqref{weak1}, we get 
\begin{equation}\label{weakl1} \frac d {dt} \|n(t)\|_{L_{1}} = \int\limits_{\Omega} f(n(t,x)) \, dx,\end{equation} so \begin{multline}\label{boun1}\frac d {dt} \|n(t)\|_{L_{1}} +\|f_-(n)\|_{L_{1}} = \|f_+(n)\|_{L_{1}} \\ \leq \int\limits_{\Omega} f(0) \, dx +  C \|n(t)\|_{L_{1}} \leq C(1+ \|n(t)\|_{L_{1}}), \end{multline} whence 
\begin{equation} \label{bound2} \|n\|_{L_{\infty}(0, T; L_{1})} + \|f_{-}(n)\|_{L_{1}(0, T; L_{1})} \leq C. \end{equation}
But \begin{equation} \|f_{+}(n)\|_{L_{1}(0, T; L_{1})} \leq C(1+ \|n\|_{L_{1}(0, T; L_{1})}) \leq C(1+ \|n\|_{L_{\infty}(0, T; L_{1})}). \end{equation}
Thus,
\begin{equation}\label{bound21} \|f(n)\|_{L_{1}(0, T; L_{1})} \leq C. \end{equation}

Putting $\theta=c^{p-1}$, $p\geq 2$, in \eqref{weak2}, we obtain $$\frac 1 p \frac d {dt} \|c(t)\|_{L_{p}}^p\leq 0,$$ and thus \begin{equation} \label{bound3} \|c\|_{L_{\infty}(0, T; L_{p})} \leq  \|c(0)\|_{L_{p}}.\end{equation}   
Passing to the limit as $p\to\infty$, we derive \begin{equation} \label{bound4} \|c\|_{L_{\infty}(0, T; L_{\infty})} \leq  C.\end{equation} Hence, 
\begin{equation} \label{bound5} \|\chi(c)\|_{L_{\infty}(0, T; L_{\infty})}+\|k(c)\|_{L_{\infty}(0, T; L_{\infty})} \leq  C.\end{equation}

Note that the fact of non-negativity of $c$ and $n$ is standard and follows from the parabolic comparison principle. 

We now take $\zeta= 1+\ln n$ in \eqref{weak1}, $\theta= -\Delta c$ in \eqref{weak2}, and $\psi= u$ in \eqref{weak3}, arriving at

\begin{equation}\label{wea1}
\frac {d}{dt} \int\limits_{\Omega} n\ln n \, dx + \frac 4 m (\nabla(n^{m/2}),\nabla(n^{m/2}))-(\chi (c) \nabla c,\nabla n)=(f(n),1+\ln n),
\end{equation}

\begin{equation}\label{wea2}
\frac 1 2 \frac {d}{dt} (\nabla c,\nabla c) + (u c,\nabla \Delta c)-(\nabla c,\nabla\Delta c) -(k(c) n,\Delta c)=0,
\end{equation}

\begin{equation}\label{wea3}
\frac 1 2  \frac {d}{dt} (u,u) + (\nabla u,\nabla u) +(n \nabla
\phi,u)=0.
\end{equation}

Integrating by parts, we rewrite \eqref{wea1} as 

\begin{multline}\label{wea11}
\frac {d}{dt} \int\limits_{\Omega} n\ln n \, dx + \frac 4 m \|\nabla(n^{m/2})\|^2 \\ +(\chi' (c) \nabla c,n \nabla c )+(\chi (c) \Delta c, n)=(f(n),1+\ln n),
\end{multline}
and observe that \begin{multline} -(u c,\nabla \Delta c)=  \sum\limits^d _ {i,j=1}  \left(\fr {u_i}{x_j} , c\, \frac{\partial^2 c}{\partial x_i \partial x_j} \right)+\left(u_i , \fr c {x_j}\, \frac{\partial^2 c}{\partial x_i \partial x_j} \right) \\ = \sum\limits^d _ {i,j=1}  \left(\fr {u_i}{x_j} , c\, \frac{\partial^2 c}{\partial x_i \partial x_j} \right)+\frac 12 \left( u_i , \fr {} {x_i}\left[\fr c {x_j} \right]^2\right)=\sum\limits^d _ {i,j=1}  \left(\fr {u_i}{x_j} , c\, \frac{\partial^2 c}{\partial x_i \partial x_j} \right). \end{multline} Now, \eqref{wea2} reads as 

\begin{equation}\label{wea21}
\frac 1 2 \frac {d}{dt} \|\nabla c\|^2 - \sum\limits^d _ {i,j=1} \left(\fr {u_i}{x_j} , c\, \frac{\partial^2 c}{\partial x_i \partial x_j} \right)+(\Delta c,\Delta c) -(k(c) n,\Delta c)=0.
\end{equation}

When $\Omega$ is bounded, due to classical regularity issues for the Neumann problem for the Poisson equation, 
\begin{equation}\label{pineq1}
\|c(t)\|_{2}\leq C (\|\Delta c(t)\|+\|c(t)\|).
\end{equation}
For the whole space, we have 
\begin{equation}\label{pineq2}
\|c(t)\|_{2}=\|c(t)-\Delta c(t)\|\leq \|c(t)\|+\|\Delta c(t)\|
\end{equation}
Hence, in both cases, \begin{equation}\label{pineq}
\|c(t)\|_{2}\leq C(\|\Delta c(t)\|+1), \ t\in (0,T).
\end{equation}

Applying \eqref{pineq} and the Cauchy inequality with epsilon  to \eqref{wea21}, we get 

\begin{equation}\label{weai1}
\frac {d}{dt} \|\nabla c\|^2 +2 K_1 \|c\|^2_2\leq C+K_2 \|\nabla u\|^2+K_3 \|n\|^2.
\end{equation}

Observe that both for $n> 1$ and $n\leq 1$ (since $f$ is $C^1$-smooth), 
\begin{equation}\label{f1}
[f(n)\ln n]_+\leq Cn |\ln n|.
\end{equation}

Therefore, \eqref{wea11} yields 
\begin{multline}\label{weai2}
\frac {d}{dt} \int\limits_{\Omega} n\ln n \, dx + \frac 4 m \|\nabla(n^{m/2})\|^2 + \|[f(n)\ln n]_-\|_{L_1} \\ \leq C+C\|n\|_{L_1}+C\|n \ln n\|_{L_1}+ K_1 \|c\|^2_2+ K_4 \|n\|^2.
\end{multline}

Multiply \eqref{wea3} by $2K_2$ and add with \eqref{weai1} and \eqref{weai2}:
\begin{multline}\label{weai3}
\frac {d}{dt} \|\nabla c\|^2 + \frac {d}{dt} \int\limits_{\Omega} n\ln n \, dx+ K_2\frac {d}{dt} \|u\|^2 \\ +K_1 \|c\|^2_2 + K_2 \|\nabla u\|^2+ \frac 4 m \|\nabla(n^{m/2})\|^2 + \|[f(n)\ln n]_-\|_{L_1} \\ \leq C+C\|n \ln n\|_{L_1}+K_5 \|n\|^2+K_6\|u \nabla \phi \|^2.
\end{multline}

If $\Omega=\R^d$, put $\zeta(x)=\wx$ in \eqref{weak1} (this test function is unbounded, but \eqref{weak1} still holds since we are dealing with strong solutions now): 
\begin{equation}\label{wea1x}
\frac {d}{dt} \|n\langle\cdot\rangle\|_{L_1}=  (u n,\nabla \langle\cdot\rangle) +(n^m,\Delta\langle\cdot\rangle)+(\chi (c) n \nabla c,\nabla \langle\cdot\rangle)+(f(n),\langle\cdot\rangle).
\end{equation}
Let us estimate the terms in the right-hand side:
\begin{equation}\label{wea1x1}
 (u n,\nabla \langle\cdot\rangle)\leq C(\|u\|^2+\|n\|^2),\end{equation}
 \begin{equation}\label{wea1x2}
 (n^m,\Delta\langle\cdot\rangle)\leq C\|n^{m/2}\|^2,\end{equation}
  \begin{equation}\label{wea1x3}
 (\chi (c) n \nabla c,\nabla \langle\cdot\rangle)\leq C(\|n\|^2+\|\nabla c\|^2),\end{equation}
 \begin{equation}\label{wea1x4}
 (f(n),\langle\cdot\rangle)\leq C\| \langle\cdot\rangle n\|_{L_{1}},\end{equation}
 whence
 \begin{equation}\label{wea2x}
3\frac {d}{dt} \|n\langle\cdot\rangle\|_{L_1}\leq K_7(1+\|u\|^2+\|n\|^2+\|n^{m/2}\|^2+\|\nabla c\|^2+\| \langle\cdot\rangle n\|_{L_{1}}).
\end{equation}
If $\Omega$ is bounded and $\wx\equiv 1$, \eqref{wea2x} is a trivial consequence of \eqref{boun1}.  

Let us show that 
\begin{equation}\label{n4}
\|n\|^2+\|n^{m/2}\|^2\leq \frac 2 {(K_5 +K_7)m} \|\nabla(n^{m/2})\|^2 + C.
\end{equation}

Indeed, let $m\leq 2$. Let $\beta=\frac 2 m$ for $d=2$ and $\beta=\frac 6 {3m-1}$ for $d=3$. In both cases $\beta<2$. Then, using the Gagliardo-Nirenberg inequality, we proceed as
\begin{multline}\label{n5}
\|n\|^2+\|n^{m/2}\|^2\leq C(\|n\|^2+\|n^{1/2}\|^2)=C(\|n^{m/2}\|^{4/m}_{L_{4/m}}+\|n\|^2_{L_1}) \\ \leq C+C\|\nabla (n^{m/2})\|^{\beta} \| n^{m/2}\|^{4/m-\beta}_{L_{2/m}}+C \| n^{m/2}\|^{4/m}_{L_{2/m}} \\ =C+C\|\nabla (n^{m/2})\|^{\beta} \| n\|^{2-m\beta/2}_{L_{1}}+C\| n\|^{2}_{L_{1}}\leq C(1+\|\nabla (n^{m/2})\|^{\beta}) \\ \leq \frac 2 {(K_5 +K_7) m} \|\nabla(n^{m/2})\|^2 + C.
\end{multline}
If $m> 2$, employing the $L_p$-interpolation, Young and Gagliardo-Nirenberg inequalities, we have
\begin{multline}\label{n7}\|n\|^2+\|n^{m/2}\|^2\leq C(\|n^{1/2}\|^2+\|n^{m/2}\|^2)\leq C(1+\|\nabla (n^{m/2})\|^{\frac {2n}{n+2}}\|n^{m/2}\|^{\frac 4{n+2}}_{L_1} +\|n^{m/2}\|^2_{L_1}) \\ \leq \frac 1 {(K_5 +K_7) m} \|\nabla(n^{m/2})\|^2+ C\|n^{m/2}\|^2_{L_1} + C= \frac 1 {(K_5 +K_7) m} \|\nabla(n^{m/2})\|^2+ C\|n\|^m_{L_{m/2}} + C \\ \leq \frac 1 {(K_5 +K_7) m} \|\nabla(n^{m/2})\|^2+ C\|n\|^{\frac {m}{m-1}}_{L_{1}}\|n\|^{\frac {m^2-2m}{m-1}}_{L_{m}} + C \\ \leq \frac 1 {(K_5 +K_7) m} \|\nabla(n^{m/2})\|^2+ C\|n^{m/2}\|^{\frac {2m-4}{m-1}} + C \\ \leq \frac 1 {(K_5 +K_7) m} \|\nabla(n^{m/2})\|^2+ \frac 1 2 \|n^{m/2}\|^{2} + C,\end{multline}
which implies \eqref{n4}.

Since $(n\ln n)_-\leq C \sqrt n$, it is easy to check (cf. \cite{ckl11} in the unbounded case) that  \begin{equation} \label{intm} \| n\ln n \|_{L_{1}}  \leq  K_8 +2 \| \langle\cdot\rangle n\|_{L_{1}}+ \int\limits_{\Omega} n\ln n \, dx.  \end{equation}  

Adding \eqref{wea2x} with  \eqref{weai3}, and taking into account \eqref{bound21},\eqref{n4} and \eqref{intm}, we get 
\begin{multline}\label{weai4}
\frac {d}{dt} \|\nabla c\|^2 + \frac {d}{dt} \int\limits_{\Omega} n\ln n \, dx+3\frac {d}{dt} \|\langle\cdot\rangle n\|_{L_1}+ K_2\frac {d}{dt} \|u\|^2 \\ +K_1 \|c\|^2_2 + K_2 \|\nabla u\|^2+ \frac 2 m \|\nabla(n^{m/2})\|^2 + \|[f(n)\ln n]_-\|_{L_1} \\ \leq C(1+ \|\nabla\phi \|^2_{L_\infty}) \\ \times\left[1+ K_8+\|\nabla c\|^2+ \int\limits_{\Omega} n\ln n \, dx+ 3\| \langle\cdot\rangle n\|_{L_{1}}+K_2\|u\|^2\right].
\end{multline}
Gronwall's inequality and \eqref{intm} yield
\begin{multline}\label{weai5} \|\nabla c\|^2+ \|n\ln n\|_{L_1}+ \| \langle\cdot\rangle n\|_{L_{1}}+K_2\|u\|^2 \\ \leq 1+ K_8+\|\nabla c\|^2+ \int\limits_{\Omega} n\ln n \, dx+ 3\| \langle\cdot\rangle n\|_{L_{1}}+K_2\|u\|^2\leq C,
\end{multline}
and \eqref{weai4} gives \begin{multline} \label{weai6} K_1 \|c\|^2_{L_2(0,T;H^2)} + K_2 \|\nabla u\|^2_{L_2(0,T;L_2)} \\ + \frac 2 m \|\nabla(n^{m/2})\|^2_{L_2(0,T;L_2)} + \|[f(n)\ln n]_-\|_{L_1(0,T;L_1)}\leq C.\end{multline}
To conclude the proof of \eqref{bound1}, it remains to remember \eqref{bound21},\eqref{f1} and \eqref{n4}.

Note that \begin{equation}\label{boundn}\|\nabla (n^m)\|_{L_1(0,T;L_1)}\leq 2 \|\nabla (n^{m/2})\|_{L_2(0,T;L_2)}\|n^{m/2}\|_{L_2(0,T;L_2)}\leq C.\end{equation}

We still require some more estimates. Firstly, let $m< 2$. We find \begin{multline}\label{smm}\|\nabla n\|_{L_{m}(0,T;L_{m})}= \|(\nabla n)^m\|_{L_{1}(0,T;L_{1})} \\ = \|(n^{m-1}\nabla n )^{2-m}(n^{\frac {m-2}2}\nabla n )^{2m-2}\|_{L_{1}(0,T;L_{1})} \\ \leq \|(n^{m-1}\nabla n )^{2-m}\|_{L_{\frac 1 {2-m}}(0,T;L_{\frac 1 {2-m}})}\|(n^{\frac {m-2}2}\nabla n )^{2m-2}\|_{L_{\frac 1 {m-1}}(0,T;L_{\frac 1 {m-1}})} \\ =\|(n^{m-1}\nabla n )\|^{2-m}_{L_{1}(0,T;L_{1})}\|(n^{\frac {m-2}2}\nabla n )\|^{2m-2}_{L_{2}(0,T;L_{2})} \\ \leq C\|\nabla (n^m) \|^{2-m}_{L_{1}(0,T;L_{1})}\|\nabla (n^{m/2}) \|^{2m-2}_{L_{2}(0,T;L_{2})}\leq C.\end{multline}

In the case $m>2$,  let $\zeta=n^{\frac {m-2}2}$ in \eqref{weak1}. Then we derive
\begin{multline}\label{wea1bm}
\frac 2 m \,\frac {d}{dt}\|n^{m/2}\|_{L_1}+ \frac {8m(m-2)}{(3m-2)^2} (\nabla(n^{\frac {3m-2}4}),\nabla(n^{\frac {3m-2}4})) \\ -\frac{m-2}{m}(\chi (c) \nabla c,\nabla (n^{m/2}))+(f_-(n),n^{\frac {m-2}2})=(f_+(n),n^{\frac {m-2}2}).
\end{multline}
Therefore, by the Cauchy-Bunyakovsky-Schwarz and Young inequalities, 
\begin{multline}\label{wea2bm}
\frac {d}{dt}\|n^{m/2}\|_{L_1}+ \|\nabla(n^{\frac {3m-2}4})\|^2+(f_-(n),n^{\frac {m-2}2}) \\ \leq C \left[\|\nabla c\|^2+\|\nabla (n^{m/2})\|^2+\|n^{m/2}\|_{L_1}+\int\limits_{\Omega} f(0)^{m/2} \, dx\right].\end{multline}
Gronwall's lemma and \eqref{bound1} imply \begin{equation} \label{boundbm}
\|n^{m/2}\|_{L_\infty(0,T;L_1)}+ \|\nabla(n^{\frac {3m-2}4})\|_{L_2(0,T;L_2)}+\|f(n)n^{\frac {m-2}2}\|_{L_1(0,T;L_1)} \leq C .
\end{equation}

We find, via a reasoning similar to \eqref{n7}, that 
\begin{equation}\label{n10}\|n^{\frac {3m-2}4}\|_{L_2(0,T;L_2)}\leq C\|\nabla(n^{\frac {3m-2}4})\|_{L_2(0,T;L_2)}+C\leq C.\end{equation}

Now, test \eqref{weak1} by the function $\zeta n^{\frac {m-2}2}$, $\zeta\in W^1_\infty$: 
\begin{multline}\label{wea3bm}
\frac 2 m \left(\frac {d}{dt} n^{m/2}, \zeta \right)+ \frac {8m(m-2)}{(3m-2)^2} (\nabla(n^{\frac {3m-2}4}),\zeta\nabla(n^{\frac {3m-2}4}))+ \frac {4m}{3m-2} (\nabla(n^{\frac {3m-2}4}),n^{\frac {3m-2}4}\nabla \zeta) \\ -\frac{m-2}{m}(\chi (c) \nabla c,\zeta\nabla (n^{m/2}))-(\chi (c) \nabla c,n^{m/2}\nabla\zeta)=(f(n),\zeta n^{\frac {m-2}2}).\end{multline}

Using \eqref{bound1}, \eqref{boundbm}, \eqref{n10}, it is easy to deduce from \eqref{wea3bm} that 
\begin{equation}\label{boundtbm}\int\limits_0^T\left|\frac {d}{dt} (n^{m/2},\zeta)\right|\,dt\leq C\|\zeta\|_{W^1_\infty}.\end{equation}

In the same manner, not necessarily for $m>2$, we derive from \eqref{bound1}, \eqref{boundn}, \eqref{weak1}--\eqref{weak3} that \begin{equation}\label{boundt1}\int\limits_0^T\left|\frac {d}{dt} (n,\zeta)\right|\,dt\leq C\|\zeta\|_{W^1_\infty},\end{equation}
\begin{equation}\label{boundt2}\int\limits_0^T\left|\frac {d}{dt} (c,\theta)\right|^2\,dt\leq C\|\theta\|_1^2,\end{equation}
\begin{equation}\label{boundt3}\int\limits_0^T\left|\frac {d}{dt} (u,\psi)\right|\,dt\leq C\|\psi\|_1.\end{equation}
Note that \eqref{boundt1} coincides with \eqref{boundtbm} for $m=2$. 

Having bounds \eqref{bound1},  \eqref{boundn}, \eqref{boundbm}, \eqref{boundtbm}--\eqref{boundt3} in hand, we can prove the existence of weak solution via approximation of \eqref{ksn1}--\eqref{ksn5} by a more regular problem, and consequent passage to the limit. We omit a major part of the details (see \cite{ckl11,flm10,dlm10,tw12} for similar issues), and restrict ourselves on the peculiarities of passage to the limit in the porous-medium-like and growth terms. For definiteness, we consider the case of bounded $\Omega$ (the unbounded case is very similar, merely the spaces $L_p$ should be replaced by $L_{p,loc}$). 

The growth term $f$ can be approximated by a sequence of bounded functions $f_N=\frac{fN}{|f|+N}$, $N\in \mathbb{N}$. Let $(n_N,c_N, u_N)$ be the corresponding sequence of solutions and $(n,c, u)$ be the limit (intended to be the weak solution). 

Due to \eqref{bound1}, without loss of generality (passing to a subsequence, if necessary) $n_N^{m/2}\to n^{m/2}$ weakly in $L_2(0,T; H^1)$. Assume first that $m\geq 2$.  In view of \eqref{boundtbm}, we can employ the Aubin--Lions--Simon lemma \cite{sim} to get $n_N^{m/2}\to n^{m/2}$ strongly in $L_2(0,T; L_2)$. On the other hand, for $m< 2$, $n_N\to n$ weakly in $L_{m}(0,T;W^1_{m})$ in view of \eqref{smm} and \eqref{bound1}. Due to \eqref{boundt1}, by the Aubin--Lions--Simon lemma we conclude that $n_N\to n$ strongly in $L_{m}(0,T;L_{m})$, whence $n_N^{m/2}\to n^{m/2}$ strongly in $L_2(0,T; L_2)$ again. Hence, in both cases, $$\nabla (n_N^m)= 2 n_N^{m/2} \nabla (n_N^{m/2}) \to 2 n^{m/2} \nabla (n^{m/2})= \nabla (n^m)$$ weakly in $L_1(0,T; L_1)$.

Finally, let us show that $f_N(n_N)\to f(n)$ in $L_1(0,T; L_{1})$. By the Vitali convergence theorem, it suffices to see that $f_N(n_N)\to f(n)$ in measure on $(0,T)\times \Omega$ (here and below we always mean ``up to a subsequence'') and $|f_N(n_N)|$ are uniformly integrable. We have $n_N^{m/2}\to n^{m/2}$ in $L_2(0,T; L_{2})$, thus $n_N\to n$ a.e. in $(0,T)\times \Omega$. Therefore $$f_N(n_N)- f(n)=-\frac{f(n_N)|f(n_N)|}{|f(n_N)|+N}+f(n_N)-f(n)\to 0$$ a.e. and hence in measure. Due to \eqref{bound1}, $\|f_N(n_N)\ln n_N\|_{L_{1}(0, T; L_{1})}\leq C$. Thus,  \begin{multline}\int_{|f_N(n_N)|>M}|f_N(n_N)|\,dx\,dt  \\ \leq C \int_{|f_N(n_N)|>M}|\ln n_N|^{-1}\,dx\,dt \leq C \int_{|f(n_N)|>M}|\ln n_N|^{-1}\,dx\,dt \to 0\end{multline} as $M\to+\infty$. 

\end{proof}

\section{The subcritical case}

\begin{theorem} \label{mr1} Let $1 \leq m\leq\frac {d+1} 3$. Suppose that \begin{equation} \label{asf2} f(y)+C_f y^2\leq f(0)+Cy \end{equation} with some positive $C_f$ independent of $y\geq 0$, and the remaining assumptions of Theorem \ref{mb1} hold. Then problem \eqref{ksn1}--\eqref{ksn5} possesses a weak solution. 
\end{theorem}

\begin{proof} Let us describe the differences with the proof of Theorem \ref{mb1}. We still need to secure inequality \eqref{bound1}. Firstly, \eqref{weakl1}, apart from yielding \eqref{boun1}, gives \begin{equation}\label{boun1su}\frac d {dt} \|n(t)\|_{L_{1}} +C_f \|n\|^2  \leq C(1+ \|n(t)\|_{L_{1}}), \end{equation} whence 
\begin{equation} \label{bound2su} \|n\|_{L_{2}(0, T; L_{2})} \leq C. \end{equation}
Since $m\leq 2$, \begin{equation} \label{bound3su} \|n^{m/2}\|_{L_{2}(0, T; L_{2})} \leq C(\|n^{1/2}\|_{L_{2}(0, T; L_{2})}+\|n\|_{L_{2}(0, T; L_{2})})\leq C. \end{equation}
Thus, we do need \eqref{n4}, which only holds in the supercritical case, but instead of \eqref{weai4} we have
\begin{multline}\label{weai4su}
\frac {d}{dt} \|\nabla c\|^2 + \frac {d}{dt} \int\limits_{\Omega} n\ln n \, dx+3\frac {d}{dt} \|\langle\cdot\rangle n\|_{L_1}+ K_2\frac {d}{dt} \|u\|^2 \\ +K_1 \|c\|^2_2 + K_2 \|\nabla u\|^2+ \frac 4 m \|\nabla(n^{m/2})\|^2 + \|[f(n)\ln n]_-\|_{L_1} \\ \leq C(1+\|n\|^2+\|n^{m/2}\|^2+ \|\nabla\phi \|^2_{L_\infty}) \\ \times\left[1+ K_8+\|\nabla c\|^2+ \int\limits_{\Omega} n\ln n \, dx+ 3\| \langle\cdot\rangle n\|_{L_{1}}+K_2\|u\|^2\right].
\end{multline}
Gronwall's lemma, \eqref{bound2su}, \eqref{bound3su} and \eqref{intm} imply \eqref{weai5}, \eqref{weai6} and  \eqref{bound1}. 
\end{proof}

\begin{theorem} Let $\Omega=\R^2$, $m=1$, $f$, $\chi$ and $k$ are $C^3$-smooth, $f'(y)+|f''(y)|\leq C$ for $y\geq 0$, $\nabla \phi \in W^2_\infty$ (and independent of $t$), $n_0\in H^2$, $c_0\in H^3$, $u_0\in H^3$, and the remaining assumptions of Theorem \ref{mr1} hold. Then there exists a unique classical solution to \eqref{ksn1}--\eqref{ksn5}, satisfying  $$n\geq 0,\ c\geq 0,$$ 
\begin{equation} \label{r1}
n \in L_{\infty}(0, T; H^2)\cap L_{2}(0, T; H^3), 
\end{equation}
\begin{equation} \label{r2}
c \in L_{\infty}(0, T; H^3)\cap L_{2}(0, T; H^4),
\end{equation}
\begin{equation} \label{r3}
u \in L_{\infty}(0, T; H^3)\cap L_{2}(0, T; H^4).
\end{equation}
\end{theorem}

\begin{proof} We observe that \begin{equation} (\nabla\, f(n), \nabla n)=(f'(n) \nabla n, \nabla n)\leq C\|\nabla n\|^2, \end{equation} and \begin{multline} (\Delta \,f(n), \Delta n)= (f'(n)\Delta n, \Delta n) + (f''(n) \nabla n \Delta n, \nabla n) \\ \leq  C(\Delta n, \Delta n) + C\|\Delta n\|\|\nabla n\|_{L_4}^2\leq C\|\Delta n\|^2+C \|\Delta n\|^2\|\nabla n\|. \end{multline} Having this at hand, one may check that the blow-up criterion \begin{equation}\label{blu}\|\nabla c\|_{L_{2}(0, T; L_\infty)}=+\infty\end{equation} proven in \cite{ckl11} for $f\equiv 0$ remains valid in our situation, and at the absence of blow-up, i.e. when \begin{equation}\label{blu1}\|\nabla c\|_{L_{2}(0, T; L_\infty)}<+\infty,\end{equation}  the solution is unique and its regularity is determined by \eqref{r1}-- \eqref{r3}. The argument showing that \eqref{blu1} takes place is a slight variation of the one ending the proof of Theorem 3 in \cite{ckl11}. 
\end{proof}

\section{Attractors} In this section we study the long-time behaviour of problem \eqref{ksn1}--\eqref{ksnb}. We restrict ourselves to the supercritical case (cf. Remark \ref{remattr} below). Since we cannot establish uniqueness of the weak solutions, we treat the question via the theory of trajectory attractors. More precisely, owing mainly to technical convenience, we use our version of the theory \cite[Chapter 4]{book} instead of more classical approaches of Chepyzhov--Vishik \cite{cvb} and Sell \cite{sell}. However, we do not know if the latter ones are applicable to \eqref{ksn1}--\eqref{ksnb}. 

In order to simplify the presentation, we consider the autonomous case $$\nabla \phi \in L_\infty$$ (independent of $t$). 
However, similar results can be obtained in the non-autonomous case via employment of the more involved theory of pullback trajectory attractors developed recently in \cite{jde11}. 

We start with recalling some basic framework from \cite[Chapter 4]{book}. 

Let $E $ and $E_0 $ be Banach spaces, $E\subset E_0 $, $E $ is reflexive.
Fix some set
$$\mathcal {H} ^ + \subset C ([0, + \infty); E_0) \cap L_\infty (0, + \infty; E) $$
of solutions (strong, weak, etc.) for any given autonomous differential equation or boundary value problem. Hereafter, the set $ \mathcal {H} ^ + $ will be
called the \textit{trajectory space} and its elements will be
called  \textit{trajectories}. Generally speaking, the nature of $ \mathcal {H} ^ + $ may be
different from the just described one.

\begin{defn} A set $P\subset $ $C ([0, + \infty ); E_0) \cap L_\infty (0,
+ \infty; E) $ is called \textit{attracting} (for the trajectory
space $ \mathcal {H} ^ + $) if for any set $B\subset \mathcal {H} ^ + $ which is bounded in $ L_\infty (0, +
\infty; E) $, one has
$$\sup\limits _ {u\in B} ^ {} \inf\limits _ {v\in P} ^ {} \|T (h) u-v \| _{C ([0, +\infty);E_0)} \underset {h\to\infty} {\to} 0. $$
\end{defn}
Here $T(h)$ stands for the translation (shift) operator,
$$ T (h) (u) (t) =u (t+h). $$ 

\begin{defn} 
 A set $P\subset $ $C ([0, + \infty ); E_0) \cap L_\infty (0, +
\infty; E) $ is called \textit{absorbing} (for the trajectory space
$ \mathcal {H} ^ + $) if for any set $B\subset \mathcal {H} ^ + $ which is bounded in $ L_\infty (0, +
\infty; E) $,  there is $h \geq 0 $
such that $ T (t) B\subset P$ for all $t \geq h $.
\end{defn}

\begin{defn} A set $ \mathcal {U}\subset $ $C ([0, + \infty ); E_0) \cap L_\infty (0,
+ \infty; E) $ is called the \textit{minimal trajectory attractor} (for
the trajectory space $ \mathcal {H} ^ + $) if

i) $ \mathcal {U} $ is compact in $C ([0, + \infty); E_0) $ and bounded in
$L_\infty (0, + \infty; E) $;

ii) $T (t) \mathcal {U}=  \mathcal {U}$ for any $t \geq 0 $;

iii) $ \mathcal {U} $ is attracting;

iv) $ \mathcal {U} $ is contained in any other set satisfying conditions i), ii), iii).
\end{defn}

\begin{defn} A set $ \mathcal {A} \subset E $ is called the \textit{global
attractor} (in $E_0 $) for the trajectory space $ \mathcal {H}
^ + $ if

i) $ \mathcal {A} $ is compact in $E_0 $ and bounded in $E $;

ii) for any bounded in $ L_\infty (0, + \infty; E) $ set $B\subset
\mathcal {H} ^ + $ the attraction property is fulfilled:
$$\sup\limits _ {u\in B} ^ {} \inf\limits _ {v\in \mathcal {A}} ^ {} \|u (t)-v \| _ {E_0} \underset {t\to\infty} {\to} 0; $$

iii) $ \mathcal {A} $ is the minimal set satisfying conditions i)
and ii) (that is, $ \mathcal {A} $ is contained in every set
satisfying conditions i) and ii)). \end{defn}

\begin{prop} \label{415} Assume that there exists an absorbing set $P $ for
the trajectory space $ \mathcal {H} ^ + $, which is relatively compact in $C ([0, + \infty); E_0) $
and bounded in $L_\infty (0, + \infty; E) $. Then there exists a
minimal trajectory attractor $ \mathcal {U} $ for the trajectory
space $ \mathcal {H} ^ + $. \end{prop}

\begin {prop} \label{416} If there exists a minimal
trajectory attractor $ \mathcal {U} $ for the trajectory space $
\mathcal {H} ^ + $, then there is a global attractor $ \mathcal
{A} $ for the trajectory space $ \mathcal {H} ^ + $, and for all $t\geq 0 $ one has $ \mathcal {A} =\{\xi (t) |\xi\in\mathcal {U} \}. $
\end {prop}

Hereafter, we make the following assumptions. 

a) $\Omega$ is bounded.

b) $m>\frac {d+1} 3$. 

c) $\phi\in L_1,\ \nabla \phi\in L_\infty$.

d) $k,\chi$ and $f$ are continuously differentiable functions, $\chi'\geq 0$, $k\geq 0$, $k(0)= 0$.

e) The initial concentration of oxygen does not exceed some constant $c_\mathcal{O}$. This unusual assumption is necessary for the presence of a compact attractor, at least when $f(0)=0$. Indeed, without an assumption of this kind no compact attractor may exist due to the presence of steady-state solutions $(n\equiv 0, c\equiv c_0, u\equiv 0)$ with arbitrarily large constants $c_0$ independent of $x$.  An alternative (which we do not like) is to fix the initial oxygen concentration, and to only let $n_0$ and $u_0$ vary. 

f) There exists a positive number $\gamma$ so that
\begin{equation}\label{asfs}f(y)+2\gamma y \leq C,\ y\geq 0,\end{equation} 
\begin{equation}\label{lam1} 2 \gamma\leq K_1,\end{equation}
and 
\begin{equation}\label{lam}4 \gamma\|u\|^2\leq \|\nabla u\|^2,\ u\in V.\end{equation}

Let us specify the class of solutions to \eqref{ksn1}--\eqref{ksnb} to be considered within this section. 

\begin{defn} A triple $(n,\, c,\, u)\in L_{\infty}(0, +\infty; L_1 \times H^1 \times H)$ 
is an \textit{admissible weak solution} to problem \eqref{ksn1}--\eqref{ksnb}
if it is a weak solution on each bounded interval $[0,T]$, and it satisfies the inequalities \begin{multline}\label{ener} \|n\|_{L_{\infty}(t,t+1;L_1)}+\|n\ln n\|_{L_{\infty}(t,t+1;L_1)} \\ +\|n\|_{L_{\infty}(t,t+1;L_{\max(1,m/2)})}^{\max(1,m/2)}+\|c\|^2_{L_{\infty}(t,t+1;H^1)}+\|u\|^2_{L_{\infty}(t,t+1;H)} \\ +\|n\|^2_{L_{2}(t,t+1;L_2)}+\|n^{[\max(4,3m-2)]/2}\|_{L_{1}(t,t+1;L_1)}+\|c\|_{L_{2}(t,t+1;H^2)}^2+\|u\|_{L_{2}(t,t+1;V)}^2 \\ \leq \Gamma[1+e^{-\gamma t}(\|n(0)\|_{L_1}+\|n(0)\ln n(0)\|_{L_1} \\ +\|n(0)\|_{L_{\max(1,m/2)}}^{\max(1,m/2)}+\|c(0)\|^2_1+\|u(0)\|^2)],\end{multline} 
\begin{equation}\label{vol} \|c(t)\|_{L_{\infty}} \leq c_\mathcal{O} \end{equation} for all $t\geq \frac{\ln (\|n_0\|_{L_1})}\gamma$, where $\Gamma$ is a certain constant depending on $\nabla \phi$, $k$, $\chi$, $f$, $c_\mathcal{O}$, $\gamma$ and $m$ (it will be defined during the proof of Theorem \ref{wsol1}).
\end{defn}

As the following proposition shows, the class of admissible weak solutions is sufficiently wide.

\begin{theorem} \label{wsol1} Let $(n_0,c_0,u_0)$ be as in Theorem \ref{mb1}, and $c_0 \leq c_\mathcal{O}$. Then there exists an admissible weak solution to \eqref{ksn1}--\eqref{ksnb} satisfying the initial condition \eqref{ksn5}. \end{theorem}

\begin{proof} It suffices to formally establish \eqref{ener} and \eqref{vol} for the solutions of \eqref{ksn1}--\eqref{ksnb}, and to pass to the limit as in the proof of Theorem \ref{mb1}.  

Inequality \eqref{vol} is straightforward, giving also \eqref{bound5}.

 As a consequence of \eqref{asfs}, we have \begin{equation} \label{asssf}(f(y)+\gamma y)\ln y \leq C,\ y\geq 0 , \end{equation} and
 \begin{equation} \label{asssf1}(f(y)+\gamma y)y^p \leq C,\ y\geq 0,\end{equation} for any fixed $p>0$.

 We deduce from \eqref{weakl1} that \begin{equation}\label{wat1} \frac d {dt} \|n(t)\|_{L_{1}} +\gamma \|n(t)\|_{L_{1}} \leq C,\end{equation} so \begin{equation}\label{wat2}\|n(t)\|_{L_{1}} \leq C+e^{-\gamma t}\|n_0\|_{L_1}. \end{equation}
 
For $t\geq \frac{\ln (\|n_0\|_{L_1})}\gamma$, we have 
 \begin{equation}\label{watc}\|n(t)\|_{L_{1}} \leq C. \end{equation}

Formulas \eqref{wea11} and \eqref{asssf} imply
\begin{multline}\label{wat3}
\frac {d}{dt} \int\limits_{\Omega} n\ln n \, dx + \frac 4 m \|\nabla(n^{m/2})\|^2 +\gamma \int\limits_{\Omega} n\ln n \\ \leq C+ K_1 \|c\|^2_2+ K_{4} \|n\|^2,
\end{multline}
whereas \eqref{wea3} gives
\begin{equation}\label{wat4}
\frac 1 2  \frac {d}{dt} \|u\|^2 + \|\nabla u\|^2 \leq K_9\|n\|^2+\frac \gamma 2 \|u\|^2.
\end{equation}

Multiply \eqref{wat4} by $2K_2 e^{\gamma t}$ and add with \eqref{weai1} and \eqref{wat3} multiplied by $e^{\gamma t}$:
\begin{multline}\label{wat5}
\frac {d}{dt} [e^{\gamma t} \|\nabla c(t)\|^2] + \frac {d}{dt} \int\limits_{\Omega} e^{\gamma t} n(t,x)\ln n(t,x) \, dx+  K_2\frac {d}{dt} [e^{\gamma t} \|u(t)\|^2] \\ -\gamma e^{\gamma t} \|\nabla c(t)\|^2-K_2 \gamma e^{\gamma t} \|u(t)\|^2 \\ +K_1 e^{\gamma t}  \|c(t)\|^2_2 + K_2 e^{\gamma t}  \|\nabla u(t)\|^2+  \frac {4 e^{\gamma t}} m \|\nabla(n^{m/2})(t)\|^2 \\ \leq C e^{\gamma t} +K_{10} e^{\gamma t}  \|n(t)\|^2+K_2 \gamma e^{\gamma t} \|u(t)\|^2.
\end{multline}

Similarly to \eqref{n4}, we see that \begin{equation}\label{wat6}
\|n(t)\|^2 \leq \frac 2 {K_{10} m} \|\nabla(n^{m/2}(t))\|^2 + C.
\end{equation}

Taking into account \eqref{lam} and \eqref{lam1}, we conclude that 
\begin{multline}\label{wat7}
\frac {d}{dt} [e^{\gamma t} \|\nabla c(t)\|^2] + \frac {d}{dt} \int\limits_{\Omega} e^{\gamma t} n(t,x)\ln n(t,x) \, dx+  K_2\frac {d}{dt} [e^{\gamma t} \|u(t)\|^2] \\ +\frac {K_1 e^{\gamma t} } 2  \|c(t)\|^2_2 + \frac {K_2  e^{\gamma t} } 2 \|\nabla u(t)\|^2+  \frac {2 e^{\gamma t}} m \|\nabla(n^{m/2})(t)\|^2 \leq C e^{\gamma t}.
\end{multline}

Integration in time implies 

\begin{multline}\label{wat8}
e^{\gamma h} \|\nabla c(h)\|^2 + \int\limits_{\Omega} e^{\gamma h} n(h,x)\ln n(h,x) \, dx+  K_2 [e^{\gamma h} \|u(h)\|^2] \\ +\int\limits_0^h \frac {K_1 e^{\gamma t} } 2  \|c(t)\|^2_2 \, dt+ \int\limits_0^h \frac {K_2  e^{\gamma t} } 2 \|\nabla u(t)\|^2\, dt+  \int\limits_0^h \frac {2 e^{\gamma t}} m \|\nabla(n^{m/2})(t)\|^2\, dt \\ \leq C\int\limits_0^h  e^{\gamma t}\, dt + \|c_0\|^2_1 + \|n_0\ln n_0\|_{L_1}+  K_2 \|u_0\|^2.
\end{multline}

Therefore, 
\begin{multline}\label{wat9}
 \|\nabla c(h)\|^2 + \int\limits_{\Omega} n(h,x)\ln n(h,x) \, dx+  K_2 \|u(h)\|^2 \\ \leq C+ e^{-\gamma h} (\|c_0\|^2_1 + \|n_0\ln n_0\|_{L_1}+  K_2 \|u_0\|^2).
\end{multline}

This inequality,  \eqref{vol}, \eqref{wat2} and \eqref{intm} yield
\begin{multline}\label{wat10}
 \|c(h)\|_1^2 + \|n(h)\ln n(h)\|_{L_1}+ \|u(h)\|^2 \\ \leq C(1+e^{-\gamma h} (\|c_0\|^2_1 +\|n_0\|_{L_1} +\|n_0\ln n_0\|_{L_1}+  \|u_0\|^2)).
\end{multline}

Integrating \eqref{wat7} from $h$ to $h+1$, we find

\begin{multline}\label{wat11}
e^{\gamma (h+1)} \|\nabla c(h+1)\|^2 + \int\limits_{\Omega} e^{\gamma (h+1)} n(h+1,x)\ln n(h+1,x) \, dx+  K_2 [e^{\gamma (h+1)} \|u(h+1)\|^2] \\ +\int\limits_h^{h+1} \frac {K_1 e^{\gamma t} } 2  \|c(t)\|^2_2 \, dt+\int\limits_h^{h+1} \frac {K_2  e^{\gamma t} } 2 \|\nabla u(t)\|^2\, dt+  \int\limits_h^{h+1} \frac {2 e^{\gamma t}} m \|\nabla(n^{m/2})(t)\|^2\, dt \\ \leq C\int\limits_h^{h+1}  e^{\gamma t}\, dt + \|c(h)\|^2_1 + \|n(h)\ln n(h)\|_{L_1}+  K_2 \|u(h)\|^2.
\end{multline}

Due to \eqref{intm}, \eqref{wat6} and \eqref{wat10}, we arrive at

\begin{multline}\label{wat12}
e^\gamma \|\nabla c(h+1)\|^2 + e^\gamma \|n(h+1)\ln n(h+1)\|_{L_1}+  K_2 e^\gamma \|u(h+1)\|^2 \\ +\frac {K_1 } 2 \int\limits_h^{h+1}   \|c(t)\|^2_2 \, dt+ \frac {K_2 } 2 \int\limits_h^{h+1} \|\nabla u(t)\|^2\, dt + {K_{10} }  \int\limits_h^{h+1}  \|n\|^2\, dt \\ \leq C(1+e^{-\gamma h} (\|c_0\|^2_1 +\|n_0\|_{L_1} +\|n_0\ln n_0\|_{L_1}+  \|u_0\|^2)).
\end{multline}

Let $m > 2$. Then \eqref{wea1bm} and \eqref{asssf1} imply
\begin{multline}\label{wat13}
\frac 2 m \,\frac {d}{dt}\|n\|_{L_{m/2}}^{m/2}+\gamma  \|n\|_{L_{m/2}}^{m/2}+ \frac {8m(m-2)}{(3m-2)^2} \|\nabla(n^{\frac {3m-2}4})\|^2 \\ \leq C(1+\|\nabla c\|^2+\|\nabla (n^{m/2})\|^2).
\end{multline}
This, \eqref{wat8} and  \eqref{intm} yield (by \cite[p. 35]{cvb}) 
\begin{multline}\label{wat14}
\|n(h)\|_{L_{m/2}}^{m/2}\leq  e^{-\gamma m h/2} \|n_0\|_{L_{m/2}}^{m/2} + C \int\limits_0^{h}e^{\gamma m (t-h)/2}(1+\|c(t)\|_2^2+\|\nabla (n^{m/2})(t)\|^2)\,dt \\ \leq e^{-\gamma h} \|n_0\|_{L_{m/2}}^{m/2} + C \int\limits_0^{h}e^{\gamma (t-h)}(1+\|c(t)\|_2^2+\|\nabla (n^{m/2})(t)\|^2)\,dt \\ \leq C(1+e^{-\gamma h} (\|n_0\|_{L_{m/2}}^{m/2} +\|c_0\|^2_1 +\|n_0\|_{L_1}+\|n_0\ln n_0\|_{L_1}+  \|u_0\|^2)).
\end{multline}
Now, integration of \eqref{wat13} from $h$ to $h+1$ gives
\begin{multline}\label{wat15} \int\limits_h^{h+1} \|\nabla(n^{\frac {3m-2}4})(t)\|^2\, dt \\ \leq C(1+e^{-\gamma h} (\|n_0\|_{L_{m/2}}^{m/2} +\|c_0\|^2_1 +\|n_0\|_{L_1} +\|n_0\ln n_0\|_{L_1}+  \|u_0\|^2)).\end{multline}
Similarly to \eqref{n10}, we deduce
\begin{multline}\label{wat16} \int\limits_h^{h+1} \|n^{(3m-2)/2}(t)\|_{L_{1}}\, dt \\ \leq C(1+e^{-\gamma h} (\|n_0\|_{L_{m/2}}^{m/2} +\|c_0\|^2_1 +\|n_0\|_{L_1} +\|n_0\ln n_0\|_{L_1}+  \|u_0\|^2)).\end{multline}

In view of \eqref{wat2}, \eqref{wat10}, \eqref{wat12}, \eqref{wat14}, \eqref{wat16} and \eqref{lam}, there exists $\Gamma$ such that \eqref{ener} holds true. 
\end{proof}

We are going to construct the minimal trajectory attractor
and the global attractor for problem \eqref{ksn1}--\eqref{ksnb}. In the sequel, we assume that \begin{equation}\label{asf3} |f(y)|\leq C(y^m + 1),\ y\geq 0, \end{equation} and $m>2.$ It seems that other supercritical values of $m$ can also be treated, even without \eqref{asf3}, although $m=2$ may be troublesome.  For this purpose, one should observe that the major part of the considerations in \cite[Chapter 4]{book} and \cite{jde11} remains valid for non-reflexive $E$. 

We let  $$E=L_{m/2} \times H^1 \times H$$ and
$$E_0=W_{m/2}^{-\delta} \times H^{1-\delta} \times V_\delta ^*,$$ where $ \delta \in (0,1] $ is a fixed
number. The trajectory space $ \mathcal {H}
^ + $ is the set of all admissible weak solutions to \eqref{ksn1}--\eqref{ksnb}. It is contained in $L_\infty (0, + \infty; E) $. Moreover, without loss of generality we may assume that it is contained in $C ([0, + \infty); (W^1_\infty)^* \times (H^{1})^* \times V^*) $. By the Lions-Magenes lemma \cite[Lemma 2.2.6]{book}, $L_\infty (0, + \infty; E) \cap C ([0, + \infty); (W^1_\infty)^* \times (H^{1})^* \times V^*) \subset C_w ([0, + \infty); E)$.  Since the embedding $E\subset E_0$ is compact, $ \mathcal {H}
^ + $ lies in $C ([0, + \infty); E_0) $. 

\begin{lemma}The time derivatives of admissible weak solutions satisfy the estimate 
\begin{multline}\label{difest}\|n'\|_{L_{\frac 32 - \frac 1m}(t,t+1;W_1^{-2})}+\|c'\|_{L_{2}(t,t+1;(H^1)^*)}^2+\|u'\|_{L_{4/3}(t,t+1;V^*)}^2 \\ \leq \Psi(\|n\|_{L_{2}(t,t+1;L_2)}, \|n\|_{L_{(3m-2)/2}(t,t+1;L_{(3m-2)/2})}, \\ \|c\|_{L_{\infty}(t,t+1;H^1)},\|u\|_{L_{\infty}(t,t+1;H)},\|u\|_{L_{2}(t,t+1;V)})\end{multline} with some continuous function $\Psi$ independent of $t\geq 0$. \end{lemma}

\begin{theorem} \label{ma1}  The trajectory space $ \mathcal {H} ^ + $ possesses a minimal trajectory attractor and a global attractor. 
\end{theorem}

\begin{proof} Due to Propositions \ref{415} and \ref{416}, it suffices to find an absorbing set for
the trajectory space $ \mathcal {H} ^ + $, which is relatively compact in $C ([0, + \infty); E_0) $
and bounded in $L_\infty (0, + \infty; E) $. Consider the set $P$ of all triples $(n,c,u)\in C ([0, + \infty); E_0) \cap L_\infty (0, + \infty; E)$ such that \eqref{difest} and
 \begin{multline}\label{ener1} \|n\|_{L_{\infty}(t,t+1;L_1)}+\|n\ln n\|_{L_{\infty}(t,t+1;L_1)} \\ +\|n\|_{L_{\infty}(t,t+1;L_{\max(1,m/2)})}^{\max(1,m/2)}+\|c\|^2_{L_{\infty}(t,t+1;H^1)}+\|u\|^2_{L_{\infty}(t,t+1;H)} \\ +\|n\|^2_{L_{2}(t,t+1;L_2)}+\|n^{(3m-2)/2}\|_{L_{1}(t,t+1;L_1)} \\ +\|c\|_{L_{2}(t,t+1;H^2)}^2+\|u\|_{L_{2}(t,t+1;V)}^2\leq 2\Gamma,\end{multline}  
hold for every $t\geq 0$.

 It is an absorbing set for
the trajectory space $ \mathcal {H} ^ + $ and is
bounded in $L_\infty (0, + \infty; E) $.  By the Aubin--Lions--Simon lemma, the set $\{y|_{[0,M]}, y\in P\}$ is relatively compact in $C ([0, M]; E_0) $ for any $M>0$.  This implies (cf. \cite[p. 183]{book}) that $P$ is relatively compact in $C ([0, + \infty); E_0) $.
\end{proof}

\begin{remark} \label{remattr}
Observe that \eqref{asf2} implies \eqref{asfs} for all positive $\gamma$, in particular, for the ones at which \eqref{lam} and \eqref{lam1} hold true. Thus, one can expect existence of attractors in the subcritical case. We leave it as an open problem. 
\end{remark}
 
\bibliography{ksns}

\bibliographystyle{abbrv}
\end{document}